\documentclass[12pt]{amsart}

\usepackage{amsmath}
\usepackage{amsthm}
\usepackage{amssymb}
\usepackage{amsfonts,mathrsfs}
\usepackage{amsxtra}
\usepackage{epsfig}
\usepackage{verbatim}
\usepackage{color}
\usepackage{hyperref}

\theoremstyle{plain}
\newtheorem{theorem}[equation]{Theorem}
\newtheorem{proposition}[equation]{Proposition}
\newtheorem{lemma}[equation]{Lemma}

\newtheorem{definition}[equation]{Definition}

\theoremstyle{remark}
\newtheorem*{remark}{Remark}

\begin{document}

\title[The Nebenh\"ulle of Smooth Complete Hartogs Domains]{A note on the Nebenh\"ulle of Smooth Complete Hartogs Domains}
\author{Yunus E. Zeytuncu}
\address{Department of Mathematics, Texas A\&M University, College Station, TX 77843}
\email{zeytuncu@math.tamu.edu}
\date{}
\begin{abstract}
It is shown that a smooth bounded pseudoconvex complete Hartogs domain in $\mathbb{C}^2$ has trivial Nebenh\"ulle.
The smoothness assumption is used to invoke a theorem of D. Catlin from \cite{Catlin80}.
\end{abstract}
%\keywords{}
\maketitle

\section{Introduction}

Let $\mathbb{D}$ denote the unit disc in $\mathbb{C}$ and let $\psi(z)$ be a continuous and bounded function on $\mathbb{D}$. Let us consider the domain $\Omega$ in $\mathbb{C}^2$ defined by; 
\begin{equation*}
\Omega=\left\{(z_1,z_2)\in \mathbb{C}^2~| z_1\in \mathbb{D}; |z_2|<e^{-\psi(z_1)}\right\}.
\end{equation*}

The domain $\Omega$ is a bounded complete Hartogs domain. Moreover, it is known that (see \cite[page 129]{VlaBook}) $\Omega$ is a pseudoconvex domain 
if and only if $\psi(z)$ is a subharmonic function on $\mathbb{D}.$ In order to focus on pseudoconvex domains; 
we assume that $\psi(z)$ is a subharmonic function for the rest of the note. 

\begin{definition}[\cite{DiedFoar77a}]
The \textit{nebenh\"ulle} of $\Omega$, denoted by $N(\Omega)$, is the interior of the intersection of all pseudoconvex domains 
that compactly contain $\Omega$. We say $\Omega$ has nontrivial Nebenh\"ulle if $N(\Omega)\setminus \Omega$ has interior points.
\end{definition}

Let $\mathcal{F}$ be the set of functions $r(z)$ where $r(z)$ is a subharmonic function on a neighborhood of $\mathbb{D}$ such that $r(z)\leq\psi(z)$ on $\mathbb{D}$. 
We define the following two functions; 
\begin{align*}
R(z)&=\sup_{r\in\mathcal{F}}\left\{r(z) \right\},\\
R^*(z)&=\limsup_{\mathbb{D}\ni\zeta\to z}R(z).
\end{align*}
Note that $R^{*}(z)$ is upper semicontinuous and subharmonic on $\mathbb{D}$.

The following proposition from \cite[Theorem 1]{Shirinbekov86} gives the description of $N(\Omega)$ for $\Omega$ a complete Hartogs domain as above.
\begin{proposition}\label{description}
$N(\Omega)=\left\{(z_1,z_2)\in \mathbb{C}^2~| z_1\in \mathbb{D}; |z_2|<e^{-R^{*}(z_1)}\right\}.$
\end{proposition}

This description does not give much information about the interior of the set difference $N(\Omega)\setminus \Omega$. 
When we drop the continuity assumption on $\psi(z)$; the well known Hartogs triangle gives an example of a domain for which $N(\Omega)\setminus \Omega$ has nonempty interior. 
On the other hand, the continuity assumption is not enough to avoid this phenomena as seen in the following example from \cite{Diederich98}.\\

\noindent \textbf{Example.} Let us take a sequence of points in $\mathbb{D}$ that accumulates at every boundary point of $\mathbb{D}$ 
and let us take a nonzero holomorphic function $f$ on $\mathbb{D}$ that vanishes on this sequence. 
The function defined by $\psi(z)=|f(z)|$ is a subharmonic function and $\Omega$, defined as above for this particular $\psi$, 
is a pseudoconvex domain. On the other hand, any pseudoconvex domain that compactly contains $\Omega$ has to contain the closure of the unit polydisc 
$\mathbb{D}\times\mathbb{D}$. Therefore, $N(\Omega)\setminus \Omega$ has nonempty interior.

This example suggests to impose more conditions on $\psi$ or $\Omega$ to have trivial Nebenh\"ulle. We prove the following theorem in this note.

\begin{theorem}\label{main} 
Suppose $\Omega=\left\{(z_1,z_2)\in \mathbb{C}^2~| z_1\in \mathbb{D}; |z_2|<e^{-\psi(z_1)}\right\}$ is a smooth bounded pseudoconvex complete Hartogs domain. 
Then $N(\Omega)=\Omega$, in particular $\Omega$ does not have nontrivial Nebenh\"ulle.
\end{theorem}
Note that the smoothness assumption on the domain $\Omega$ is a stronger condition than the smoothness assumption on the function $\psi(z)$.

For the rest of the note; $\mathcal{O}(\Omega)$ denotes the set of functions that are holomorphic on $\Omega$, 
$C^{\infty}(\overline{\Omega})$ denotes the set of functions that are smooth up to the boundary of $\Omega$ and $A^{\infty}(\Omega)$ 
denotes the intersection of these two sets.

\section{Proof of Theorem \ref{main}}

Suppose $N(\Omega)\not=\Omega$ and take $p=(p_1,p_2)\in N(\Omega)\setminus \Omega$. 
By Proposition \ref{description}, we have $R^{*}(p_1)<\psi(p_1)$ and by semicontinuity of $R^{*}$ and continuity of $\psi$; 
there exists a neighborhood $\mathcal{U}$ of $p_1$ inside $\mathbb{D}$ such that for all $q_1\in \mathcal{U}$ 
we have $R^{*}(q_1)<\psi(q_1)$. The neighborhood $\mathcal{U}$ guarantees that $N(\Omega)$ contains a full neighborhood 
(in $\mathbb{C}^2$) of the the boundary point $(p_1,e^{-\psi(p_1)}) \in b\Omega$.

After this observation, we prove the following uniform estimate.
\begin{lemma}\label{uniform}
Suppose $p\in N(\Omega)$ and $f$ is a function that is holomorphic in a neighborhood of $\Omega$. Then $|f(p)|\leq\sup_{q\in \Omega}|f(q)|.$
\end{lemma}
\begin{proof}
Assume otherwise, then $g(z_1,z_2)=\frac{1}{f(z_1,z_2)-f(p)}$ is a holomorphic function on some complete Hartogs domain $\Omega_1$ that compactly contains $\Omega$. 

The domain $\Omega_1$ is not necessarily pseudoconvex but its envelope of holomorphy $\widetilde{\Omega_1}$ (which is a single-sheeted(schlicht) and complete Hartogs domain) 
is pseudoconvex (see \cite[page 183]{VlaBook}). 
Moreover, any function holomorphic on $\Omega_1$ extends to a holomorphic function on $\widetilde{\Omega_1}$. 

In particular, $g(z_1,z_2)$ is holomorphic on $\widetilde{\Omega_1}$ and therefore $p\not\in\widetilde{\Omega_1}$. But this is impossible since $p\in N(\Omega)$. This contradiction finishes the proof of the lemma.
\end{proof}

Next, we state an approximation result that is a simpler version of the one in \cite{BarrettFornaess}. Let us take a holomorphic function $f$ on $\Omega$. 
We can expand $f$ as follows: $$f(z_1,z_2)=\sum_{k=0}^{\infty}a_k(z_1)z_{2}^k,$$ where $a_k(z_1)$ is a holomorphic function on $\mathbb{D}$ for all $k\in \mathbb{N}$. 
Let us define the following functions, for any $N\in\mathbb{N}$, 
\begin{equation}
\mathcal{P}_N(z_1,z_2)=\sum_{k=0}^{N}a_k\left(\frac{z_1}{1+\frac{1}{N}}\right)z_2^k.
\end{equation}
It is clear that, each $\mathcal{P}_N$ is a holomorphic function in a neighborhood of $\overline{\Omega}$. 

\begin{lemma}\label{approximation}
Suppose $f \in A^{\infty}(\Omega)$. 
Then the sequence of functions $\left\{\mathcal{P}_N\right\}$ converges uniformly to $f$ on $\overline{\Omega}$.
\end{lemma}

\begin{proof}
For $(z_1,z_2)\in \Omega$ and $k\geq 2$, we have;
\begin{align*}
|a_k(z_1)z_2^k|&=\left|\frac{1}{k!}\frac{(k-2)!}{2\pi i}z_2^k\int_{|\zeta|=e^{-\psi(z_1)}}\frac{\frac{\partial^2}{\partial \zeta^2}f(z_1,\zeta)}{\zeta^{k-1}}d\zeta\right|\\
&\leq \frac{1}{2\pi k(k-1)}\left(e^{-\psi(z_1)}\right)^k2\pi e^{-\psi(z_1)}\left(\sup_{\Omega}\left|\frac{\partial^2}{\partial z_2^2}f\right|\right)\frac{1}{(e^{-\psi(z_1)})^{k-1}}\\
&\leq \frac{C}{k^2}
\end{align*}
for some global constant $C$. This gives the uniform convergence.
\end{proof}

Since each $\mathcal{P}_N$ is holomorphic on a neighborhood of $\Omega$; in particular it is holomorphic on $N(\Omega)$. By Lemma \ref{uniform}, the uniform convergence
percolates onto $N(\Omega)$ and therefore we get a holomorphic extension of any function in $A^{\infty}(\Omega)$ to $N(\Omega)$. 
On the other hand, let us remember the following theorem from \cite{Catlin80}.
\begin{theorem}[Catlin, \cite{Catlin80}]\label{Catlin}
On any smooth bounded pseudoconvex domain there exists a function in $A^{\infty}(\Omega)$ that does not extend holomorphically to a neighborhood
of any boundary point.
\end{theorem}
In the first paragraph of this section we showed if $N(\Omega)\not=\Omega$ then $N(\Omega)$ contains a full neighborhood of a boundary point. 
This observation with the one in the previous paragraph contradict Theorem \ref{Catlin}. Therefore we conclude the proof of Theorem \ref{main}.\\   

\begin{remark}
In the description of $\Omega$, the base domain is assumed to be the unit disc $\mathbb{D}$. However, the result \ref{main} is true when the base
is any other planar domain $D$. The description in Proposition \ref{description} and the remark about the envelope of holomorphies, in the proof of Lemma \ref{uniform}, 
are also valid for any base $D$. The approximation statement in Lemma \ref{approximation} can be modified for any base $D$, see \cite{BarrettFornaess}.
\end{remark}

\begin{remark}
Note that $N(\Omega)=\Omega$ does not imply that $\Omega$ has a Stein neighborhood basis; see \cite[Proposition 1]{Sato80} for a false proof and 
\cite{Stensones} for a counterexample.
\end{remark}

%\section*{Acknowledgments} I'd like to thank S. Sahutoglu for useful remarks and discussion on this paper.\\

\bibliographystyle{plain}
\bibliography{SteinNbhd}

\end{document}